\documentclass[a4paper,12pt]{amsart}
\usepackage{mathrsfs}

\usepackage{amsfonts}
\usepackage{amsmath}
\usepackage{amssymb}
\usepackage{enumerate}
\usepackage{hyperref}
\usepackage{latexsym}
\usepackage{color}
\usepackage{cite}
\def\a{\alpha}
\def\b{\beta}

\newcommand\ZZ{\mathbb{Z}} 
 
\newcommand\Aut{\mathrm{Aut}}  
\newcommand\Cay{\mathrm{Cay}}

\newcommand\GL{\mathrm{GL}}

\newtheorem{theorem}{Theorem}[section]
\newtheorem{corollary}[theorem]{Corollary}

\newtheorem{lemma}[theorem]{Lemma}
\newtheorem{proposition}[theorem]{Proposition}
\theoremstyle{definition}
\newtheorem{definition}[theorem]{Definition}

\begin{document}
\title{Isomorphism factorizations of the complete graph into Cayley graphs on CI-groups}

\author[abbreviation]{Huye Chen, Jingjian Li, Hao Yu, Zitong Yu}

\address{Guangxi Center for Mathematical Research   \\
\& Center for Applied Mathematics of Guangxi (Guangxi University) \\
Guangxi University \\
Nanning, Guangxi 530004, P.R. China.}

\email{chenhy280@gxu.edu.cn (H. Chen); \allowbreak lijjhx@gxu.edu.cn (J.J. Li); \newline  haoyu@gxu.edu.cn (H. Yu);
\allowbreak   yuzitong2020@163.com (Z. Yu)}

\begin{abstract}

Isomorphic factorizations of complete graphs originate from the seminal work of Frank Harary and collaborators, who initiated the systematic study of decompositions of complete graphs into pairwise isomorphic spanning subgraphs.
In this paper, we investigate isomorphic factorizations of complete graphs into Cayley graphs on CI-groups.

Let $\Gamma=\Cay(G,S)$ denote the Cayley graph of finite group $G$. We obtain a necessary and sufficient condition on CI-group $G$ so that the complete graph on $|G|$ vertices can be edge-partitioned into  $k$-copies of Cayley graph of the same CI-group $G$ each isomorphic to $\Cay(G,S)$  for some inverse-closed subset $S\subset G\setminus\{1\}$.
Further we give a construction of isomorphic factorizations of the complete graph into Cayley graphs on CI-group.

\end{abstract}

\maketitle \textit{Key words:} isomorphism factorization; CI-group; Cayley graph

\section{Introduction}
All graphs $\Gamma = (V, E)$ considered in this paper are simple, undirected graphs. The vertex set and edge set of $\Gamma$ are denoted by $V(\Gamma)$ and $E(\Gamma)$, respectively.
The automorphism group of $\Gamma$, denoted by $\operatorname{Aut}(\Gamma)$, is the subgroup of $\operatorname{Sym}(V(\Gamma))$ that preserves the edge set $E(\Gamma)$.
A \textit{partition} of a set is a collection of non-empty, pairwise disjoint subsets whose union is the entire set.
Let $\Gamma$ be a graph, and let $\mathcal{P}:=\{P_0,P_1,\cdots,P_{k-1}\}$ be a partition of edge set $E(\Gamma)$ with $k\geq 2$. Then $(\Gamma,\mathcal{P})$ is called a \textit{decomposition} of the graph $\Gamma$.
Moreover, if the subgraph induced by each of the set $P_i$ are all spanning subgraph of $\Gamma$, then the decomposition $(\Gamma,\mathcal{P})$ is called a \textit{factorization} of the graph $\Gamma$, and each spanning subgraph $\Gamma_i$ is called a \textit{factor} of $\Gamma$, where $0\leq i\leq k-1$.
A factorization $(\Gamma, \mathcal{P})$ is termed \textit{isomorphic factorization} if all its factors are pairwise isomorphic. For such a factorization, we denote the common factor (up to isomorphism) by $\Gamma/k$. In the specific case of a complete graph $K_n$, each factor $K_n/k$ in an isomorphic factorization is called a $k$-if graph (where `if' abbreviates isomorphic factorization).

We note that a $2$-if graph is precisely a self-complementary graph, a topic of extensive study (see, e.g., \cite{CO,L,LP2001,LR,LRS,Wetal,LP2000}). A key motivation for investigating self-complementary graphs stems from their applications to the study of Ramsey numbers, a classic problem in graph theory.

Research on isomorphic factorizations is grounded in the seminal series spanning \cite{HR1} to \cite{HR10}. A principal direction that has emerged concerns problems with prescribed symmetries. Typical conditions require $\operatorname{Aut}(\Gamma)$ to act transitively on $\mathcal{P}$, or $\operatorname{Aut}(\Gamma) \cap \operatorname{Aut}(\Gamma_i)$ to act transitively on a factor's edge set, vertex set, or arc set. Work such as \cite{BPP,BBM,BL,CK,DD,FLW,GLP,Metal2007,Metal2008,LP,LP03,L2006,LLP,M,P2008,PLS,X2017} exemplifies this direction of research.


Let $G$ be a finite group, and set $G^*:=G\setminus\{1\}$. For a subset $S$ of $G^*$ satisfying $S^{-1}=S$, the \textit{Cayley graph} $\Gamma=\Cay(G,S)$ is defined as the graph with vertex set $V(\Gamma)=G$ and edge set $E(\Gamma)=\{\{h,g\}|g,h\in G,\,gh^{-1}\in S\}$. And Cayley graph $\Gamma=\Cay(G,S)$ is connected if and only if $\langle S\rangle=G$.
 It is clear that $\Cay(G,S)\cong\Cay(G,S^\sigma)$ for any $\sigma\in\Aut(G)$.
 Let $G$ be a finite group. An inverse-closed subset $S\subset G^*$ is called a \textit{CI-set} of $G$ if, for every inverse-closed subset $T\subset G^*$, an isomorphism $\Cay(G,S)\cong\Cay(G,T)$ implies that $T=S^{\alpha}$ for some $\alpha \in \Aut(G)$. If every inverse-closed subset of $G^*$ is a CI-set, then $G$ is called a \textit{CI-group}.

In their studies of isomorphic factorizations of complete digraphs with prime order, the authors of \cite{CO, CL} constructed both self-complementary and 3-if Cayley graphs. Also, \cite{Wetal} provided a construction method for 2-if Cayley graphs, based on the orbits of group actions. In this paper, we give a construction of $k$-if Cayley graphs on
a $k$-rotational group (see Section 2 for the definiton), and classify all the CI-groups which admit a $k$-if Cayley graphs for some positive integer $k\geq 2$.
A finite group $G$ is said to have \textit{$k$-if property} if there exists a Cayley graph $X=\Cay(G,S)$ of $G$ is a $k$-if Cayley graph.
So, the main result of this paper is determine the CI-groups, which has $k$-if property.
We now state the main result of this paper, which follows from Theorems \ref{Case1}, \ref{Z489}, \ref{Case2}, and \ref{Case3}, is given in Section 4.
\begin{theorem}\label{mainCI}
	Let $G$ be a CI-group. Then $G$ has $k$-if property if and only if $G$ is the direct product of elementary abelian group, and the order of each Sylow subgroup $G_p$ of $G$ satisfy $2k\mid |G_p|-1$ when $p$ is odd; $k\mid |G_p|-1$ when $p=2$.
\end{theorem}

We now give an overview of the article. In Section 2, we provide some preliminary results.
In Section 3, we introduce a special group, called the $k$-rotational group, which generalizes the notion of a reflexible group introduced in \cite{FRS}. Using this group, we present a construction of $k$-if Cayley graphs on a $k$-rotational group.
 In the final section, we investigate the existence of $k$-if Cayley graphs on CI-groups.

\section{Preliminaries}
In this paper, we will consider all CI-groups and discuss their $k$-if property. In 1996, Li \cite{DocLiCH} gave the classifications of CI-groups.
For positive coprime integers $l$ and $m$, we denotes the least positive integer $e$ such that $l^e\equiv 1\pmod m$ by $o(l\pmod m)$. We say a group is \textit{homocyclic} if it is a direct product of cyclic groups of the same order.

\begin{definition}\cite[Definition 11.3.1 ]{DocLiCH}
\label{He}
	Let $M$ be abelian group in which all Sylow subgroups are homocyclic, $m$ be the exponent of $M$, $l$ satisfy $1\leq l \leq m$ and $(l(l-1),m)=1$. Set $e=o(l\pmod m)$. Let $r,\, s$ be nonnegative integers and suppose that one of the following holds:
	\begin{table}[h]
		\centering
		\begin{tabular}{|c|c|c|c|c|}
			\hline
			type $ e $ & $ r $ & $ s $ & $ m $ \\
			\hline
			2    &  $\geq 1$ & 0   & $ m $ odd \\
			\hline
			3    &  0   & $\geq 1$ & $ 3 \nmid m $ \\
			\hline
			4    &  $\geq 2$ & 0   & $ m $ odd \\
			\hline
			6    &  $\geq 1$ & $\geq 1$ & $(m, 6) = 1$ \\
			\hline
		\end{tabular}
	\end{table}

\noindent Define $H_e(M,2^r3^s,l):=\langle M,z|z^{2^r3^s}=1,\forall x\in M,x^z=x^l\rangle$ and $H_e(M,2^r3^s,l)$ is called of type $e$.
\end{definition}
It is worth noting that any CI-group $G$ is necessarily an $m$-CI group for every integer $m$ with $1 \leq m \leq |G|$. By combining \cite[Corollary 12.5.2]{DocLiCH} and \cite[Theorem 12.5.1]{DocLiCH}, we arrive at the following proposition.

	\begin{proposition}
		\label{CIgroup}
		Let $G$ be a CI-group, then each Sylow subgroup of $G$ is either an elementary abelian group or isomorphic to one of the following: $\ZZ_4$, $\ZZ_8$, $\ZZ_9$ or $Q_8$. Moreover, $G=U\times V$, where $(|U|,|V|)=1$, $U$ is an abelian group, and $V$ is either trivial or one of the following:
		\begin{enumerate}
			\item $Q_8$, $\ZZ^2_2\rtimes\ZZ_3$, $\ZZ^2_2\rtimes \ZZ_9$, $Q_8\rtimes \ZZ_3$, $Q_8\rtimes \ZZ_9$, $\ZZ^2_3\rtimes Q_8;$
			\item $H_e(M,2^r3^s,l)$ where $M$ is abelian, $r$, $s\geq 0$ and $r+s\geq 1;$
			\item $H_2(M,2^r,l)\times H_3(M',3^s,l')$, where $|M|,2,|M'|$ and $3$ are pairwise coprime; $M$, $M'$ are abelian, $1\leq r\leq 3$, $1\leq s\leq 2$ and $l, l' > 1$.
		\end{enumerate}
	\end{proposition}

\vskip 5mm

Before considering the $k$-if property of CI-groups, we first prove the following lemmas.

\begin{lemma}\label{toDP}
Suppose $G = H \times K$ and let $A \subset H^*$, $B \subset K^*$. If both $\Cay(H, A)$ and $\Cay(K, B)$ are $k$-if Cayley graphs, then $\Cay(G, AK \cup B)$ is a $k$-if Cayley graph.
\end{lemma}

\begin{proof}
Suppose $\Cay(H,A)$ and $\Cay(K,B)$ are $k$-if Cayley graphs.
Then there are partitions $\mathcal{A}=\{A_0=A,A_1,\cdots ,A_{k-1}\}$  of $H^*$ and $\mathcal{B}=\{B_0=B,B_1,\cdots ,B_{k-1}\}$ of $K^*$ such that $\Cay(H,A_i) \cong \Cay(H,A)$ and $\Cay(K,B_i) \cong \Cay(K,B)$ for all $1\leq i\leq k-1$.

	Set $X_i:=\Cay(G,A_iK\cup B_i)$, where $i\in\ZZ_k$. We claim that $\Cay(G,A_0K\cup B_0)\cong \Cay(G,A_iK\cup B_i)$, for each $i\in\ZZ_k$.
	Suppose $\varphi$ is an isomorphism from $\Cay(H,A_0)$ to $\Cay(H,A_i)$, and $\psi$ is an isomorphism from $\Cay(K,B_0)$ to $\Cay(K,B_i)$.
 Now, for any $h\in H$ and $k\in K$, we define a mapping $\rho: V(X_0)\to V(X_i)$ as follows:
$\rho(hk)= h^{\varphi}k^{\psi}.$
Clearly, $\rho$ is a bijection.
We now prove that $\rho$ is an isomorphism between $X_0$ and $X_i$. To this end, it suffices to check that $\rho$ preserves the adjacency relation: for any $u, v \in V(X_0)$, $\{u,v\}\in E(X_0)\Leftrightarrow \{\rho(u),\rho(v)\}\in E(X_i).$	
	For each $\{hk,h_1k_1\}\in E(X_0)$, where $h\neq h_1$. Then $(h_1k_1)(hk)^{-1}=(h_1h^{-1})(k_1k^{-1})\in A_0K$, hence
	$h_1^{\varphi}k_1^{\psi}(h^{\varphi}k^{\psi})^{-1}=(h_1)^{\varphi}(h^{\varphi})^{-1}(k_1)^{\psi}(k^{\psi})^{-1}\in A_iK$.
	For each $\{hk,hk_1\}\in E(X_0)$, then $hk_1(hk)^{-1}=k_1k^{-1}\in B_0$, hence
	$h^{\varphi}k_1^{\psi}(h^{\varphi}k^{\psi})^{-1}=(k_1)^{\psi}(k^{\psi})^{-1}\in B_i$, as desired.

Set $S_i := A_i K \cup B_i$ for $i \in \mathbb{Z}_k$. We then prove that the sets $S_i$ are pairwise disjoint and partition $G^*$. This, together with the above arguments, shows that $X_i$ is a $k$-if Cayley graph.
	Recall that $A_i\cap A_j=B_i\cap B_j=\emptyset$ for $i\neq j$, then $A_iK\cap A_jK=\emptyset$.
	In fact, we have
	$$
	S_i\cap S_j=(A_iK\cup B_i)\cap(A_jK\cup B_j)=(A_iK\cap A_jK)\cup(B_i\cap B_j)=\emptyset \quad( 0\leq i,j\leq k-1),
	$$
	and the disjoint union of $S_i$ is
	$$
	\bigsqcup\limits_{i=0}^{k-1}S_i=\bigsqcup\limits_{i=0}^{k-1}(A_iK\cup B_i)=(\bigsqcup\limits_{i=0}^{k-1}A_iK)\cup(\bigsqcup\limits_{i=0}^{k-1}B_i)=(H^*)K\cup K^*=G^*,
	$$
	as desired.	
\end{proof}

A \textit{characteristic subgroup} $H$ of $G$ is a subgroup of $G$ satisfying $H^{\alpha}=H$ for each $\alpha\in\Aut(G)$.

\begin{proposition}\cite[Lemma 2.3.9]{DocLiCH}\label{CharacterCI}
	Let $H$ be a characteristic subgroup of a CI-group $G$. Then $H$ is a CI-group.
\end{proposition}

Let $S_1$ and $S_2$ be two inverse-closed subsets of $G^*$. We say that $S_1$ and $S_2$ are \textit{equivalent}, denoted by $S_1\sim S_2$, if there exists an automorphism $\alpha\in \Aut(G)$ such that $S_1^{\alpha}=S_2$. By the definition of a CI-group, it follows that for any CI-group $G$, the Cayley graphs $\Cay(G,S_1)$ and $\Cay(G,S_2)$ are isomorphic if and only if $S_1\sim S_2$.
By the definition of a $k$-if graph, we have that for a CI-group $G$, a Cayley graph $\Cay(G,S)$ is a $k$-if graph if and only if there exists a partition  $\{S_0=S, S_1,\cdots,S_{k-1}\}$ of $G^*$ such that $S^{-1}=S$ and $S_i\sim S_j$ for any $i,j\in\ZZ_k.$
Let $n(G,l):=|\{x\in G\mid o(x)=l\}|$.
If $G$ is a CI-group, which has $k$-if property, then $2k\mid n(G,l)$ $(l\neq 2)$ and $k\mid n(G,2)$.
\begin{lemma}\label{subChar}
Let $G$ be a CI-group. If $G$ has the $k$-if property, then every nontrivial characteristic subgroup $H$ of $G$ also has the $k$-if property.
\end{lemma}
\begin{proof}
	Suppose $G$ is a CI-group with the $k$-if property. Then there exists a partition $\{S_0=S,S_1,\cdots, S_{k-1}\}$ of $G^*$ such that $S\sim S_i\,(1\leq i\leq k-1)$.
Take $A_i=H\cap S_i$ for $i=0,1,\dots,k-1$.
	Then $|A_i|=|A_j|\neq 0$ for all $0\leq i,j\leq k-1$, since $H$ is a characteristic subgroup of $G$.
	One easily verifies that for $i \neq j$,
	$$
	A_i\cap A_j=(H\cap S_i)\cap(H\cap S_j)=H\cap S_i\cap S_j=\emptyset,
	$$
	and
	$$
	\bigsqcup\limits_{i=0}^{k-1}A_i=H\cap\bigsqcup\limits_{i=0}^{k-1}S_i=H\cap(G^*)=H^*.
	$$
	Thus $\{A_0,A_1,\cdots,A_{k-1}\}$ is a partition of $H^*$.
	
	In addition, $A_i^{-1}=(H\cap S_i)^{-1}=H^{-1}\cap S_i^{-1}=H\cap S_i=A_i$.
Since  $S_0 \sim S_i$ for each $i$, we have $A_0^{\alpha}=(H\cap S_0)^{\alpha}=H^{\alpha}\cap S_0^{\alpha}=H\cap S_i=A_i$, which implies that $A_0\sim A_i$.
By Proposition \ref{CharacterCI}, $H$ is also a CI-group.
Consequently, $\Cay(H,A)=\Cay(H,A_0)$ is a $k$-if Cayley graph.
\end{proof}

\section{$k$-rotational group}
From the definitions of $k$-if graphs and Cayley graphs, we obtain that if there exists an automorphism $\sigma \in \Aut(G)$ such that $\{S, S^{\sigma}, \cdots, S^{\sigma^{k-1}}\}$ partitions $G^*$ with $S=S^{-1}$, then $\Gamma = \Cay(G,S)$ is a $k$-if graph.  This property provides a powerful tool for generating examples of $k$-if graphs.
The case $k=2$ has been studied in \cite{S1985,FRS}, where such a group $G$ is termed {\it reflexible}. To facilitate further discussion, we introduce the following generalization.
 A group $G$ is said to be {\it $r$-rotational} if there exists an automorphism $\sigma \in \Aut(G)$ such that $\{S, S^{\sigma}, \cdots, S^{\sigma^{r-1}}\}$ partitions $G^*$ with $S=S^{-1}$.
Under this definition, a reflexible group is exactly a $2$-rotational group.
And we have that a $k$-rotational group $G$ has $k$-if property.

We present a construction (Lemma~\ref{Costr}) for $k$-if Cayley graphs, which generalizes \cite[Lemma 9]{Wetal}.  Based on this, Theorem~\ref{LiCostr} then characterizes all $r$-rotational group when $r$ is prime.

\begin{lemma}\label{Costr}
Let $k\geq 2$ be an integer. Assume that $G$ has a fixed-point-free automorphism $\sigma$, namely, $\sigma$ fixed no non-identity elements of $G$, and $o(\sigma)=k^e$ for some positive integer $e$; moreover, if $k$ is even, we require that $\sigma^2$ is also fixed-point-free, except when $2k=|G^*|$.
Then $G$ is $k$-rotational.
In particular, we define a subset $S$ as follows:
\begin{enumerate}
	\item Let $\Delta_0,\Delta_1,\cdots,\Delta_{n-1}$ be the $\langle \sigma\rangle$-orbit on $\Omega:=\{\{g,g^{-1}\}|g\in G^*\}$, and label the $k$ orbits of $\langle \sigma^k \rangle$ on $\Delta_i$ as $\Delta_{i0},\Delta_{i1},\cdots,\Delta_{i(k-1)}$ for $i\in \ZZ_n$.
	\item Set $S=\bigcup_{i\in \ZZ_n}\Delta_{i0}$.
\end{enumerate}
 Then $\{S,S^{\sigma},\cdots,S^{\sigma^{k-1}}\}$ is a partition of $G^*$. Moreover, $S^{\sigma^i}=(S^{\sigma^i})^{-1}$ for each $i\in\ZZ_n$.
\end{lemma}

\begin{proof}
	To ensure the existence of $\Delta_{ij}$ for all $i \in \mathbb{Z}_n$ and $j \in \mathbb{Z}_k$, we first prove that $|\Delta_i| > 1$ for each $i \in \mathbb{Z}_n$. This condition guarantees that every orbit of $\langle \sigma \rangle$ on $\Omega$ can be partitioned into $k$ distinct parts by the action of the subgroup $\langle \sigma^k \rangle$.
	
If $|G^*| = 2k$, then $|G|$ is odd and $|\Omega| = k$. Given that $\sigma$ is fixed-point-free and $|\sigma| = k^e$, the subgroup $\langle \sigma \rangle$ acts transitively on $\Omega$. Therefore, we must have $n = e = 1$ and $|\Delta_0| = k$

For the case where $|G^*| \neq 2k$, suppose, for contradiction, that there exists a $\langle\sigma\rangle$-orbit $\Delta := \{{g, g^{-1}}\}$ of size $1$ in $\Omega$. Since $\sigma$ is fixed-point-free, we must have $g^{\sigma} \neq g$, which forces $g^{\sigma} = g^{-1}$ and consequently $g^{\sigma^{2}} = g$.
We now consider two cases. If $g$ is an involution, then $g^{\sigma} = g^{-1} = g$, which would imply that $k$ is even. This contradicts the assumption that $\sigma^2$ is fixed-point-free.
If $g$ is not an involution, then we examine the parity of $k$. If $k$ is odd, then $g = g^{\sigma^{k^e}} = g^{\sigma} = g^{-1}$, which is impossible. If $k$ is even, then $g^{\sigma^2} = g$, again contradicting the fact that $\sigma^2$ is fixed-point-free.

Therefore, each $\Delta_i$, the orbit of $\langle \sigma\rangle$ acting on the set $\Omega:=\{\{g,g^{-1}\}|g\in G^*\}$, has size greater than $1$.
It then follows that for each $l \in \ZZ_k$, we have $S^{\sigma^l} = \bigcup_{i \in \ZZ_n} \Delta_{il}$. Hence, the collection $\{S, S^{\sigma}, \dots, S^{\sigma^{k-1}}\}$ forms a partition of $G^*$.
Therefore, $G$ is $k$-rotational. Moreover, the property $S^{\sigma^i} = (S^{\sigma^i})^{-1}$ holds for each $i \in \ZZ_n$, which follows directly from the fact that $\Delta_{ij} = \Delta_{ij}^{-1}$ for all $i \in \ZZ_n$ and $j \in \ZZ_k$.
\end{proof}

\begin{lemma}\label{LiCostr}
	Let $G$ be a finite group, and $r$ be a prime. If $r$ is odd, then $G$ is $r$-rotational if and only if $G$ has an automorphism $\sigma$ of order a power of $r$ such that $\sigma$ is fixed-point-free.
If $r=2$, then $G$ is $2$-rotational if and only if $G$ has an automorphism $\sigma$ of order a power of $r$ such that both $\sigma$ and $\sigma^2$ is fixed-point-free.

\end{lemma}
\begin{proof}
	The sufficiency has been treated by Lemma \ref{Costr}.
	Now, we assume that $G$ is $r$-rotational. Then $\{S,S^{\tau},S^{\tau^2},\cdots,S^{\tau^{r-1}}\}$ is a partition of $G^*$ for some $\tau\in\Aut(G)$ and $S^{-1}=S$. Then $\tau$ is fixed-point-free.
	Since $\langle\tau\rangle/\langle\tau^r\rangle\cong\ZZ_r$, we have $o(\tau)=r^em$ for some positive integers $e$ and $m$, with $(m,r)=1$.
	Therefore, set $\sigma:=\tau^m\in \Aut(G)$ which is of order $r$-power, and we claim that $\{S,S^{\tau},S^{\tau^2},\cdots,S^{\tau^{r-1}}\}=
	\{S,S^{\sigma},S^{\sigma^2},\cdots,S^{\sigma^{r-1}}\}$. Notice that $S^{\tau^{x_1}}=S^{\tau^{x_2}}$ if and only if $x_1\equiv x_2\pmod r$. Since $(r,m)=1$, we easy know that $m$ is a generating elements in $\ZZ_r$. Hence $\{S,S^{\tau},S^{\tau^2},\cdots,S^{\tau^{r-1}}\}=
	\{S,S^{\tau^m},S^{\tau^{2m}},\cdots,S^{\tau^{(r-1)m}}\}=
	\{S,S^{\sigma},S^{\sigma^2},\cdots,S^{\sigma^{r-1}}\}$. So, $\sigma$ is fixed-point-free.
If $r=2$, we have $\sigma^2$ is fixed-point-free from Lemma 3 in \cite{S1985}.
\end{proof}

\begin{lemma}\label{RtoDP}
Let $G=H\times K$. Then $G$ is $k$-rotational if $H$ and $K$ are $k$-rotational for some positive integer $k\geq 2$.
\end{lemma}
\begin{proof}
Suppose that $H$ and $K$ are $k$-rotational. Then there exist $\a\in \Aut(H),\b\in\Aut(K)$ such that $\{A,A^{\a},\cdots,A^{a^{k-1}}\}$ and $\{B,B^{\b},\cdots,B^{\b^{k-1}}\}$ are partitions of $H^*$ and $K^*$ respectively; and $A=A^{-1}$, $B=B^{-1}$.
Take $S=AK\cup B$, $\sigma=\a\b\in\Aut(H)\times\Aut(K)\subseteq\Aut(G)$. Then $S^{-1}=S$. For any $i,j\in\ZZ_k$, we have $$S^{\sigma^{i}}\cap S^{\sigma^{j}}=(AK\cup B)^{\a^{i}\b^{i}} \cap  (AK\cup B)^{\a^{j}\b^{j}}=(A^{\a^{i}}K\cup B^{\b^i})\cap(A^{\a^j}K\cup B^{\b^j})=\emptyset$$
	and
	$$\bigsqcup\limits_{i\in \ZZ_k} S^{\sigma^i}=\bigsqcup\limits_{i\in\ZZ_k} (AK\cup B)^{\a^i\b^i}=\bigsqcup\limits_{i\in \ZZ_k} A^{\a^i}K\cup B^{\b^i}=H^*K\cup K^*=G^*.$$
	Thus $\{S,S^{\sigma},S^{\sigma^2},\cdots,S^{\sigma^{k-1}}\}$ is a partition of $G^*$. Hence, $G$ is $k$-rotational.
\end{proof}

\begin{lemma}\label{RtoChar}
	If $G$ is $k$-rotational, then each characteristic subgroup $H$ of $G$ is $k$-rotational provided $H\neq1$.
\end{lemma}
\begin{proof}
	Suppose $G$ is $k$-rotational. Then there exists a partition $\{S,S^{\sigma},S^{\sigma^2},\cdots,S^{\sigma^{k-1}}\}$ of $G^*$ for some $\sigma\in\Aut(G)$ and $S^{-1}=S$.
Take $A_i=H\cap S^{\sigma^{i}}$, where $i\in\ZZ_k$.
	Now, we have that $A_0=A_0^{-1}$ and $|A_i|=|A_j|\neq 0$ with any $0\leq i,j\leq k-1$, as $H$ is a characteristic subgroup of $G$.
It easy to show that
	$$
	A_i\cap A_j=(H\cap S^{\sigma^{i}})\cap(H\cap S^{\sigma^{j}})=H\cap S^{\sigma^{i}}\cap S^{\sigma^{j}}=\emptyset,
	$$
	and
	$$
	\bigsqcup\limits_{i=0}^{k-1}A_i=H\cap\bigsqcup\limits_{i=0}^{k-1}S^{\sigma^{i}}=H\cap G^*=H^*.
	$$
	So $A_0,A_1,\cdots,A_{k-1}$ is a partition of $H^*$. In addition, $A_i^{-1}=(H\cap S^{\sigma^{i}})^{-1}=H^{-1}\cap (S^{\sigma^{i}})^{-1}=H\cap S^{\sigma^{i}}=A_i$.
As $H$ is a characteristic subgroup of $G$, we have $A_0^{\sigma^{i}}=(H\cap S)^{\sigma^{i}}=H^{\sigma^{i}}\cap S^{\sigma^{i}}=H\cap S^{\sigma^{i}}=A_i$, which implies that $\{A=A_0,A^{\sigma},\cdots,A^{\sigma^{k-1}}\}$ is a partition of $H^*$, as desired.
\end{proof}

\begin{corollary}
	Let $G$ be CI-group. Then $G$ is $k$-rotational if and only if $G$ has $k$-if property.
\end{corollary}
\begin{proof}
This result followed from Theorem \ref{mainCI} and Theorem \ref{Case1}.
\end{proof}

\section{Proof of Theorem \ref{mainCI}}
This section is devoted to the proof of Theorem \ref{mainCI}. Following Proposition~\ref{CIgroup}, we divide CI-groups into the following three cases and treat them separately:

Case 1: $G=U$, $U$ is an abelian group;

Case 2: $G=U\times V$, $(|U|,|V|)=1$; $U$ is an abelian group and $V\cong A\in\{Q_8, \ZZ^2_2\rtimes\ZZ_3, \ZZ^2_2\rtimes \ZZ_9, Q_8\rtimes \ZZ_3, Q_8\rtimes \ZZ_9, \ZZ^2_3\rtimes Q_8\};$

Case 3: $G=U\times V$ with $(|U|,|V|)=1$. Here, $U$ is an abelian group, and $V$ is either isomorphic to $H_e(M,2^r3^s,l)$, where $M$ is abelian, $r,s\geq 0$ and $r+s\geq 1$; or isomorphic to $H_2(M,2^r,l)\times H_3(M',3^s,l')$, where $|M|,2,|M'|$ and $3$ are pairwise coprime, $M$ and $M'$ are abelian, $1\leq r\leq 3$, $1\leq s\leq 2$ and $l, l' > 1$.
\vskip 3mm
Throughout this section, we assume that $k \geq 2$.
\subsection{Case 1}
\quad
\vspace{0.5em}

According to Proposition \ref{CIgroup}, every Sylow subgroup of $G$ is either elementary abelian or isomorphic to $\ZZ_4$, $\ZZ_8$ or $\ZZ_9$.
The corresponding results are given in Theorem \ref{Case1} and Theorem \ref{Z489}, respectively.

\begin{lemma}\label{Zpn}
  Let $G\cong \ZZ_p^n$ be an elementary abelian $p$-group. Then the following statements are equivalent:

  (1) $G$ is $k$-rotational;

  (2) $G$ has $k$-if property;

  (3) $2k\mid p^n-1$, when $p$ is odd; $k\mid p^n-1$, when $p=2$.
\end{lemma}

\begin{proof}
Obviously,
if $\ZZ_p^n$ is $k$-rotational, then we get that $\ZZ_p^n$ has $k$-if property, followed from the definitions.	
If $\ZZ_p^n$ has $k$-if property, then we have $2k\mid (|Z_p^n|-1)=p^n-1$, when $p$ is odd;  $k\mid p^n-1$, when $p=2$.

Suppose now that $2k$ divides $p^n - 1$ and that $p$ is odd. Recall that $\GL(n, p)$ embeds naturally into $\Aut(\mathbb{Z}_p^n)$. In this setting, the Singer subgroup $T \leq \GL(n, p)$ is cyclic of order $p^n - 1$ and acts regularly on the set $G^* = \mathbb{Z}_p^n \setminus \{0\}$. Consequently, there exists an automorphism $\sigma \in T$ of order $k^e$ for which both $\sigma$ and $\sigma^2$ act fixed-point-freely on $G^*$. Applying Lemma \ref{Costr}, we conclude that $G$ is $k$-rotational.
An analogous argument establishes the $k$-rotational property for $G$ when $p = 2$.
\end{proof}

\begin{theorem} \label{Case1}
Let $G$ be an abelian CI-group with all Sylow subgroups elementary abelian. Then $G$ admits a direct product decomposition $G=\ZZ_{p_1}^{r_1}\times\ZZ_{p_2}^{r_2}\times\cdots\times \ZZ_{p_n}^{r_n}$, with $p_1, p_2, \cdots, p_n$ distinct primes. For such a group, the following are equivalent:
	
	(1) $G$ is $k$-rotational;
	
	(2) $G$ has $k$-if property;
	
	(3) for all $1\leq i \leq n$, if $p_i$ is odd, then $2k\mid (p_i^{r_i}-1)$; if $p_i=2$, then $k\mid (p_i^{r_i}-1)$.
\end{theorem}
\begin{proof}
If $G$ is $k$-rotational, then $G$ has $k$-if property.
	
Suppose $G$ possesses the $k$-if property. Since each $\ZZ_{p_i}^{r_i}$ is a characteristic subgroup of $G$, Lemma \ref{subChar} implies that $\ZZ_{p_i}^{r_i}$ also has the $k$-if property for every $i = 1, \dots, n$. This yields that $2k\mid (|\ZZ_{p_i}^{r_i}|-1)=(p_i^{r_i}-1)$ when $p_i$ is odd; $k\mid (p_i^{r_i}-1)$ when $p_i=2$.

We now prove  $(3)\Rightarrow(1)$ by induction on $n$. For $n=1$, $G=\ZZ_{p_1}^{r_1}$ is a $k$-rotational group by Lemma \ref{Zpn}.
Set $G_0:=\ZZ_{p_1}^{r_1}\times\ZZ_{p_2}^{r_2}\times\cdots\times \ZZ_{p_{n-1}}^{r_{n-1}}$ be a subgroup of $G$. By the induction hypothesis, $G_0$ is $k$-rotational. Also, $\ZZ_{p_n}^{r_n}$ is $k$-rotational. According to Lemma \ref{RtoDP}, $G$ is $k$-rotational, as $G\cong G_0\times \ZZ_{p_n}^{r_n}$.
This completes this proof.
\end{proof}

\begin{theorem}\label{Z489}
    Let $G$ be an abelian CI-group. If $G$ has a Sylow subgroup isomorphic to $\ZZ_4$, $\ZZ_8$ or $\ZZ_9$, then $G$ has no $k$-if property.
\end{theorem}
\begin{proof}
    Since every Sylow subgroup of an abelian group is characteristic, it suffices to show that $\mathbb{Z}_4$, $\mathbb{Z}_8$, and $\mathbb{Z}_9$ themselves do not have the $k$-if property, by Lemma \ref{subChar}.
Since $\ZZ_l$ has only one involution for $l\in\{4,8\}$, we have that $\ZZ_4$ and $\ZZ_8$ have no $k$-if property for $k>1$.
    For $\ZZ_9$, there exists only two elements of order $3$ in $\ZZ_9$ and these two elements are inverse to each other. This implies that $k = 1$, contradicting the assumption that $k \geq 2$.
\end{proof}

Combining Proposition~\ref{CIgroup}, Theorems~\ref{Case1} and \ref{Z489}, we get the following corollary.

\begin{corollary}\label{Abelian=ZZZ}
    Let $G$ be an abelian CI-group with $k$-if property. Then every Sylow subgroup of $G$ is elementary abelian.
\end{corollary}

\subsection{Case 2}
\quad
\vspace{0.5em}

In this part, we consider the case that the CI-group $G=U\times V$, $(|U|,|V|)=1$; $U$ is an abelian group and $V\cong A\in\{Q_8, \ZZ^2_2\rtimes\ZZ_3, \ZZ^2_2\rtimes \ZZ_9, Q_8\rtimes \ZZ_3, Q_8\rtimes \ZZ_9, \ZZ^2_3\rtimes Q_8\}$.
Followed from Theorem~\ref{Z489} that we only need to consider the case that the Sylow subgroups of $U$ are elementary abelian group.
The main result in this section is,

\begin{theorem}\label{Case2}
Let $G=U\times V$ be a CI-group with $(|U|,|V|)=1$, where $U$ is abelian and all the Sylow subgroups of $U$ are elementary abelian group, and $V\cong A\in\{Q_8, \ZZ^2_2\rtimes\ZZ_3, \ZZ^2_2\rtimes \ZZ_9, Q_8\rtimes \ZZ_3, Q_8\rtimes \ZZ_9, \ZZ^2_3\rtimes Q_8\}$. Then $G$ has no $k$-if property.
\end{theorem}

To prove the theorem above, we require the following lemma.

\begin{lemma}\label{p+1Grp}
Let $G$ be an abelian group that is not an elementary abelian $2$-group. If $|G|-1$ is an odd prime, then $G$ has no $k$-if property.
\end{lemma}
\begin{proof}
For contradiction, suppose that $G$ has $k$-if property and that $|G|-1=p$ is an odd prime. Then $G^*$ possesses a unique partition $\{\{g_1\},\{g_2\},\cdots,\{g_p\}\}$ of size greater than $1$. Hence, for each $g_i\in G^*$, $\Cay(G,g_i)$ is a $p$-if Cayley graph, which can happen only if $g_i^2=1$, a contradiction.
\end{proof}

The above lemma now allows us to prove Theorem \ref{Case2}.
\begin{proof}
	Let $G$ be a CI-group as defined in Theorem \ref{Case2}. Then $V$ is a characteristic subgorup of $G$.
 We will show that $V$ has no $k$-if property, from which Theorem \ref{Case2} follows by Lemma \ref{subChar}.
For the cases where $V\ncong\ZZ^2_2\rtimes \ZZ_9$, we have $|V|-1=p$, an odd prime. By Lemma \ref{p+1Grp}, this implies that $V$ has no $k$-if property.
	Now consider the case $V\cong\ZZ^2_2\rtimes \ZZ_9$. Since $|V|-1 = 35$, the group $V$ can have the $k$-if property only for $k\in \{5, 7, 35\}$.
Since $V$ has only three elements of order $2$, the possible values of $k$ are restricted to $1$ and $3$. This contradicts the requirement that $k \in \{5,7,35\}$, so $V$ cannot have the $k$-if property.
\end{proof}

\subsection{Case 3}
\quad
\vspace{0.5em}

The main result of this section is the following theorem.

\begin{theorem}\label{Case3}
	Let $G=U\times V$ be a CI-group with $(|U|,|V|)=1$, where $U$ is an abelian group and
 $V$ is one of the following types:
\begin{enumerate}
	\item $H_e(M,2^r3^s,l)$, where $M$ is abelian, $r$, $s\geq 0$ and $r+s\geq 1;$
	\item $H_2(M,2^r,l)\times H_3(M',3^s,l')$, where $|M|,2,|M'|$ and $3$ are pairwise coprime; $M$ and $M'$ are abelian, $1\leq r\leq 3$, $1\leq s\leq 2$ and $l, l' > 1$.
\end{enumerate} Then $G$ has no $k$-if property.
\end{theorem}
\begin{proof}
Since $(|U|,|V|)=1$ and $V$ is normal in $G$, it follows that $V$ is a characteristic subgroup of $G$. Furthermore, every subgroup $H_e(M,2^r3^s,l)$ of type $e$ is a characteristic subgroup of $V$. Since Lemma \ref{Vnoif} establishes that every subgroup of the form $H_e(M,2^r3^s,l)$ (for any type $e$) has no $k$-if property, Lemma \ref{subChar} implies that $G$ also has no $k$-if property.
\end{proof}


\begin{lemma}\label{Z=Z_2346}
	Let $H_e(M,2^r3^s,l):=\langle M,z|z^{2^r3^s}=1,\forall x\in M,x^z=x^l\rangle$ is of type $e$, as defined in Definition \ref{He}. If $H_e(M,2^r3^s,l)$ is a CI-group and has $k$-if property, then $\langle z\rangle\cong \ZZ_e$ with the type $e\in\{2,3,4,6\}.$
\end{lemma}
\begin{proof}
Set $H_e := H_e(M,2^r3^s,l) = M \rtimes Z$ as in Definition \ref{He}, where $Z \cong \mathbb{Z}_{2^r3^s}$. Note that $|M|$ is odd when $r \geq 1$, and $3 \nmid |M|$ when $s \geq 1$. Hence $(|M|,|Z|)=1$.
	Every element $v \in H_e$ can be uniquely expressed as $v = z^t x$ for some $x \in M$ and integer $t$ with $1 \leq t \leq |Z|$.
    Now assume that $\ZZ_2\lesssim Z$ or $\ZZ_3\lesssim Z$. Then $k\mid n(H_e,2)$ and $2k\mid n(H_e,3)$. Based on this fact, we will prove $\langle z\rangle\cong\ZZ_e$ by determining $n(H_e,3)$ and $n(H_e,2)$.

For all $x\in M$, we have $o(x)\neq 2$ when $e=2,4,6$, and $o(x)\neq 3$ when $e=3,6$.
Let $z^a x,\ z^b y \in H_e$, where $a,b$ are integers, $x,y \in M$, and suppose that $o(z^a x)=2$ and $o(z^b y)=3$.
 Then $(z^ax)^2=z^{2a}x^{l^a+1}=1$ and $(z^by)^3=z^{3b}y^{l^{2b}+l^b+1}=1$.
We now prove our result by analyzing the type $e$ of the group $H_e$.

\vskip 3mm	
	Case $e=2:$
\vskip 3mm
    Then we get $Z\cong \mathbb{Z}_{2^r}$ by Definition \ref{He}. If $r>1$, then $a=2^{r-1}$.
    Recall that $l^e\equiv 1\pmod{m}$. Consequently, $l^{2^{r-1}}+1=(l^2)^{2^{r-2}}+1\equiv1^{2^{r-2}}+1\equiv 2\pmod{m}$, which implies that $x^2=1$.
    This forces $x=1$ and consequently $n(H_e,2)=1$, as $(|M|,|Z|)=1$. This contradicts to the $k$-if property of $H_e$. Hence, $Z\cong \mathbb{Z}_2$.

\vskip 3mm	
    Case $e=3:$
\vskip 3mm
	Then we have $Z\cong \mathbb{Z}_{3^s}$ by Definition \ref{He}.
If $s>1$, then $b=\pm3^{s-1}$.  From $l^e \equiv 1 \pmod{m}$, we deduce $l^{2b}+l^b+1=(l^3)^{\pm2\cdot3^{s-2}}+(l^3)^{\pm3^{s-2}}+1\equiv1^{\pm2\cdot3^{s-2}}+1^{\pm3^{s-2}}+1\equiv 3\pmod{m}$, yielding $y^3=1$.
This implies that $y=1$ and $n(H_e,3)=2$, as $(|M|,|Z|)=1$. This contradicts to the $k$-if property of $H_e$. Consequently, $Z\cong \mathbb{Z}_3$.

\vskip 3mm	
    Case $e=4:$
\vskip 3mm
	Then we get $Z\cong \mathbb{Z}_{2^r}$ with $r\geq 2$, by Definition \ref{He}.  If $r>2$, then $a=2^{r-1}$. Thus, $l^{2^{r-1}}+1=(l^4)^{2^{r-3}}+1\equiv1^{2^{r-3}}+1\equiv 2\pmod{m}$, which implies that $x^2=1$.
 This implies that $x=1$ and $n(H_e,2)=1$, as $(|M|,|Z|)=1$. This contradicts to the $k$-if property of $H_e$. Consequently, $Z\cong \mathbb{Z}_4$.

\vskip 3mm	
	Case $e=6:$
\vskip 3mm
By Definition \ref{He}, we have  $Z\cong \mathbb{Z}_{2^r3^s}$ with $r,s\geq 1$.
     If $r>1$ or $s>1$, then either $a=2^{r-1}3^s$ or $b=\pm2^r3^{s-1}$. Consequently,
   $l^{2^{r-1}3^s}+1=(l^6)^{2^{r-2}3^{s-1}}+1\equiv1^{2^{r-2}3^{s}-1}+1\equiv 2\pmod{m}$ or $l^{2b}+l^b+1=(l^6)^{\pm2^{r}\cdot3^{s-2}}+(l^6)^{\pm2^{r-1}3^{s-2}}+1\equiv1^{\pm2^{r}\cdot3^{s-2}}+1^{\pm2^{r-1}3^{s-2}}+1\equiv 3\pmod{m}$, which implies that $x^2=1$ or $y^3=1$, respectively.
    This implies that $x=1$ and $n(H_e,2)=1$; or $y=1$ and $n(H_e,3)=2$, as $(|M|,|Z|)=1$. This contradicts to the $k$-if property of $H_e$. Consequently, $Z\cong \mathbb{Z}_6$.
	
	The proof is complete.
\end{proof}

\begin{lemma}\label{Vnoif}
	Let $H_e(M,2^r3^s,l)=M\rtimes Z$ be a CI-group, where $Z=\langle z\rangle\cong\ZZ_{2^r3^s}$. Then $H_e(M,2^r3^s,l)$ has no $k$-if property.
\end{lemma}
\begin{proof}
Let $m$ be the exponent of $M$.  The hypothesis that $H_e:=H_e(M,2^r3^s,l)$ is a CI-group, together with Proposition \ref{CIgroup}, implies that the Sylow subgroups of $H_e$ are all either elementary abelian or isomorphic to  $\ZZ_4,\ZZ_8,\ZZ_9, Q_8$. Following the proof of Lemma~\ref{Z=Z_2346}, we have that $M$ is a characteristic subgroup of $H_e$. On the contrast, we assume that $H_e$ has $k$-if property.
Then $M$ is a CI-group, which has $k$-if property. Hence, $k\mid n(M,2)$ and $2k\mid n(M,3)$. Further, we have that all the Sylow subgroups of $M$ are elementary abelian, by Lemma \ref{Z489}.
By Lemma \ref{Z=Z_2346}, we get $Z\cong \ZZ_e$ for each group $H_e$ of type $e$.
In what follows, we derive a contradiction based on this assumption by considering the cases $e \in \{2,3,4,6\}$ separately.

\vskip 3mm
    Case $e=2$ and $Z\cong\ZZ_2:$
\vskip 3mm
    In this case,  $l^2 \equiv 1 \pmod{m}$. Since $l^2 - 1 = (l-1)(l+1)$, it follows that for each prime divisor $p$ of $m$, either $l \equiv 1 \pmod{p}$ or $l \equiv -1 \pmod{p}$.
    It follows that $M$ can be written as  $M =M_1 \times M_2$, where $x^z=x$ for any $x\in M_1$ and $x^z=x^{-1}$ for any $x\in M_2$.
    Then $M_1$ centralizes $Z$, giving $H_e=M_1\times(M_2\rtimes Z)$.
     Note that $x^z=x^{-1}$ for any $x\in M_2$, which implies that $(zx)^2=1$. Consequently, we obtain $n(H_e,2)=|M_2|$ and $k\mid |M_2|$.
     We claim that $M_2$ is a characteristic subgroup of $H_e$.
     Indeed,
      for any $y\in M_2$ and $\sigma\in\Aut(H_e)$, there exists $y_0\in M_2$ such that $z=(y_0z)^\sigma$. Then $(y^{\sigma})^zy^{\sigma}=(z^{-1}y_0^{-1})^\sigma y^\sigma (y_0z)^\sigma y^\sigma=1$, which implies that $y^{\sigma}\in M_2$, as desired.

Lemma \ref{subChar} implies that $M_2$ has $k$-if property.
      Then for any prime divisor $p$ of $|M_2|$, we have $2k\mid p^s-1$ (with $p^s$ the largest prime-power divisor of $|M_2|$).
        Recall that $k\mid |M_2|$, so $k=1$, a contradiction.
\vskip 3mm
    Case $e=3$ and $Z\cong\ZZ_3:$
\vskip 3mm
In this case, $l^3\equiv 1\pmod{m}$. Since $l^3-1=(l-1)(l^2+l+1)$, it follows that for every prime divisor $p$ of $m$, either $l\equiv1 \pmod p$ or
$l^2+l\equiv -1\pmod p$.
Let $M_2$ be defined as $M_2:=\langle x\mid x^z\neq x,x\in M\rangle$.
For any $x\in M_2$, we have $(zx)^3=z^{3}x^{l^{2}+l+1}=z^{3}x^{-1+1}=z^{3}=1$, so $zx$ is of order $3$. A similar calculation gives $(z^2x)^3=z^{6}x^{l^{4}+l^2+1}=x^{l^4+l^2+1}$, and the congruence $l^2+l+1\equiv 0\pmod{p}$ forces $l^4+l^2+1\equiv l^4+(-l-1)+1\equiv l(l^3-1)=0\pmod{p}$, whence $z^2x$ also has order $3$. Consequently, $n(H_e,3)=2|M_2|$, and therefore $2k\mid 2|M_2|$.

By the same arguments as in the first case, we conclude that $M_2$ is a characteristic subgroup of $H_e$. Consequently, $M_2$ has $k$-if property.
It follows that for every prime $p$ dividing $|M_2|$, $2k$ divides $p^s - 1$, where $p^s$ is the largest prime-power divisor of $|M_2|$.
Together with the fact that $k \mid |M_2|$, we obtain $k = 1$, a contradiction.
\vskip 3mm
    Case $e=4$ and $Z\cong\ZZ_4:$
\vskip 3mm
	In this case, $l^4\equiv 1\pmod{m}$. Since $l^4-1=(l-1)(l+1)(l^2+1)$, it follows that for every prime divisor $p$ of $m$ either $l\equiv \pm1\pmod{p}$ or $l^2\equiv -1\pmod p$.
Hence, $M$ admits a subgroup $M_2$, where $(x_2^z)^z=x_2^{l^2}=x_2^{-1}$ for all $x_2 \in M_2$.
Let $z^ix$ be an involution for $x\in M$ and $1\leq i\leq 3$. Then $(z^ix)^2=z^{2i}x^{l^i+1}=1$, thus $i=2$, and then $x^{l^2+1}=1$. This implies that $x\in M_2$.
Then, by the same arguments as in the first case, we have that $M_2$ is a  characteristic subgroups of $H_e$, and have $k$-if property.  So we have $k\mid n(H_e,2)=|M_2|$, and $2k\mid p^s-1$, where $p^s$ is the largest prime-power divisor of $|M_2|$. Therefore, we have $k=1$, a contradiction.
\vskip 3mm
    Case $e=6$ and $Z\cong\ZZ_6:$
\vskip 3mm
	In this case, $l^6\equiv 1\pmod{m}$. Since $l^6-1=(l-1)(l+1)(l^2+l+1)(l^2-l+1)$, $M$ has a subgroup $M_1\times M_2\times M_3$, where the factors are characterized by the following congruences: $l\equiv -1\pmod{o(x_1)}$ for all $x_1\in M_1$; $(l^2+l)\equiv -1\pmod{o(x_2)}$ for all $x_2\in M_2$; and $(l^2-l)\equiv -1\pmod{o(x_3)}$ for all $x_3\in M_3$.
	Let $z^ix$ be an involution for $1\leq i\leq 5$. Then $(z^ix)^2=z^{2i}x^{l^i+1}=1$, thus $i=3$, and then $x^{l^3+1}=x^{(l+1)(l^2-l+1)}=1$ implies that $x\in M_1\times M_3$. So we have $k\mid n(H_6,2)=|M_1\times M_3|$.
By the same arguments as in the first case, we have that $M_1\times M_3$ is a characteristic subgroup of $H_e$.
  Then $2k\mid p^s-1$, where $p^s$ is the largest prime-power divisor of $|M_1\times M_3|$. Therefore, we have $k=1$, a contradiction.
	
In summary, the CI-group  $H_e=H_e(M,2^r3^s,l)$ has no $k$-if property for integers $k\geq 2$.	
\end{proof}

\vskip 3mm
\begin{center}{\large\bf Acknowledgements}\end{center}
\vskip 2mm
The first author thanks the supports of the National Natural Science Foundation of China (12301446).
The second author thanks the supports of Natural Science Foundation of Guangxi (2025GXNSFAA069013) and the National Natural Science Foundation of China (12571362).

\end{document}